\documentclass[12pt,twoside, reqno]{amsart}
\usepackage{amsmath,amsopn}
\usepackage{hyperref}
\usepackage{amsfonts}
\usepackage{amssymb}
 \newtheorem{thm}{Theorem}[section]
 \newtheorem{exm}[thm]{Example}

 \newtheorem{lemma}[thm]{Lemma}
 
 \theoremstyle{definition}
 \newtheorem{definition}[thm]{Definition}
 \newtheorem{rem}[thm]{Remark}
\emergencystretch=\maxdimen
\begin{document}
\begin{center}
\title[]{}
\textbf{OSCILLATION CRITERIA FOR HIGHER ORDER NONLINEAR GENERALIZED NEUTRAL DIFFERENCE EQUATIONS} \\

\textbf{  Adem K{\i}l{\i}\c{c}man$^{1}$, P. Venkata Mohan Reddy$^{2}$, M. Maria Susai Manuel$^{3}$. }\\

$^{1}$Department of Mathematics and Institute for Mathematical Research,\\ University Putra Malaysia, 43400 UPM, Serdang, Selangor, Malaysia \\

$^{2,3}$ Department of Science and Humanities, R. M. D. Engineering College, Kavaraipettai - 601 206, Tamil Nadu, S. India.\\
 
 \markboth{OSCILLATION CRITERIA FOR HIGHER ORDER NONLINEAR GENERALIZED NEUTRAL  DIFFERENCE EQUATIONS}{Adem K{\i}l{\i}\c{c}man, P.Venkata Mohan Reddy, M.Maria Susai Manuel. }
\end{center}
\textbf{Abstract}\\
\indent\qquad In the present study we highlight some results related to the oscillation for high order nonlinear generalized neutral difference equation in the following form   
\begin{equation*}
	\Delta_{\ell}\left(a(k)\Delta_{\ell}^{m-1}z(k)\right)+q(k)f(x({k-\rho\ell}))=0, \label{1e}
\end{equation*}
 where $z(k)=x(k)+p(k)x({\tau(k)})$.\\
 \textbf{Keywords and phrases:} Oscillation, generalized difference equation, higher order, nonlinear, generalized neutral difference equation.  \\  
 \textbf{AMS   Subject Classification:}\quad  39A10, 39A11.
\section{Introduction}
\indent \quad The difference equations are defined based on the operator $\Delta$ given in the form of
$\displaystyle{\Delta u(k) = u(k+1)-u(k)}$ for $ \ k\in \mathbb{N}=\{0,1,2,3,\cdots\}$. Many researchers made significant contribution  in the related studies, see (\cite{agar00}, \cite{ronal90}-\cite{wal91}). However, in the development of $\Delta$ 
\begin{equation}{\label{c02}}
\Delta u(k) = u(k+\ell)-u(k), \ k \in \mathbb{R}, \ \ell \in \mathbb{N}(1),
\end{equation}
we can not see too much progress. \\

\indent\qquad Recently many researchers studying the oscillatory and asymptotic behavior of solutions of higher order neutral type involving $\Delta$, as these equations naturally arise in several applications including in population dynamics. Here, we generalize the results obtained earlier via the generalized difference operator $\Delta_{\ell}$ for the generalized neutral difference equations involving the operator $\Delta_{\ell}$. Many authors studied to find sufficient conditions that ensure all solutions including the bounded solutions of neutral type are oscillatory, see (\cite{agar01}, \cite{ruba01}, \cite{thanda01}, \cite{yildiaz01} - \cite{zhou01}). But none have attempted to generalize these results using the operator $\Delta_{\ell}$ for the neutral type. \\

 In this study, we consider 
  \begin{equation}
 \Delta_{\ell}\left(a(k)\Delta_{\ell}^{m-1}z(k)\right)+q(k)f(x({k-\rho\ell}))=0, \ k\geq k_0 \label{2e}
 \end{equation} and we have identified sufficient conditions for the solutions to be oscilatory so that either of the following true
\begin{equation*}
 \sum_{r=0}^{\infty}\frac{1}{a(k_0+j+r\ell)}=\infty ~~\text{or}~~ \sum_{r=0}^{\infty}\frac{1}{a(k_0+j+r\ell)}<\infty.
\end{equation*}

\section{Some Preliminary Requisites}
\indent\qquad In this section, we recall some definitions, lemmas and theorems that will be useful during the development of the study. 
\begin{definition}\cite{Msm.gb06}
 Let $u(k)$, $k\in[0,\infty)$ be a real or complex valued function and $\ell\in(0,\infty)$. Then, the operator  $\Delta_{\ell}$ is defined as
\begin{equation}\label{rep1}
 \Delta_{\ell} u(k) \  = \  u(k+\ell)-u(k).
\end{equation}
and is called generalized difference operator and $r^{th}$ order is given by
\begin{equation}\label{rep2}
\Delta_{\ell}^ru(k)=\underbrace{\Delta_\ell(\Delta_\ell(\dots(\Delta_\ell u(k))))}_{r\ times}.
\end{equation}
\end{definition}
 \begin{definition}\cite{Msm.gb06}\label{inv}
Let $u(k)$ be a real or complex valued function and $\ell\in(0,\infty)$. Then, the inverse of operator $\Delta_\ell$ denoted by $\Delta_\ell^{-1}$ is defined as
\begin{equation}\label{rep5}
  \text{If  }\Delta_{\ell}v(k)\  =\  u(k), \ \text{ then }  v(k)\  =\  \Delta_{\ell}^{-1}u(k)+c_j ,
\end{equation}
where $c_j$ is a constant for all $k\in\mathbb{N}_\ell(j),j=k-\left[\frac{k}{\ell}\right]\ell$.
\end{definition}
\begin{lemma}\cite{Msm.gb06}\label{sumo}
If a real valued function $u(k)$ is defined for $k\in[k_0,\infty)$, then 
\begin{equation}\label{2}
 \Delta_\ell^{-1}u(k)=\sum\limits_{r=1}^{\left[\frac{k-k_0}{\ell}\right]}u(k-r\ell)+c_j, 
\end{equation} where $c_j$ is a constant for $k\in\mathbb{N}_\ell(j)$, $j=k-k_0-\Big[\frac{k-k_0}{\ell}\Big]\ell$. 
\end{lemma}
\begin{thm}\label{sum}
If $\Delta_\ell v(k)=u(k)$ for $k\in[k_0,\infty)$ and $j=k-k_0-\Big[\frac{k-k_0}{\ell}\Big]\ell$, then
\[v(k)-v(k_0+j)=\sum\limits_{r=0}^{\left[\frac{k-k_0-j-\ell}{\ell}\right]}u(k_0+j+r\ell).\]
\end{thm}
\begin{proof}
The proof follows from Definition \ref{inv},and Lemma \ref{sumo} and $c_j=v(k_2+j)$.
\end{proof}
\begin{lemma}\cite{mmssubc}\label{lem1.8.4}
Let $u(k)$ and $v(k)$ be defined on 
$[k_0,\infty)$. Then, 
\begin{equation}{\label{1.8.4}}
\Delta_{\ell}\{u(k)v(k)\}\ =\ u(k+\ell)\Delta_{\ell}v(k)\ +v(k)\Delta_{\ell}u(k)\ =u(k)\Delta_{\ell}v(k)+ v(k+\ell)\Delta_{\ell}u(k)
 \end{equation} 
 for all $k\in [k_0,\infty)$ and $j=k-k_0-\left[\frac{k-k_0}{\ell}\right]\ell$.
\end{lemma}

\begin{lemma}\label{cor1.8.6}\cite{mmssubc}
Let $u(k)$ be defined on $[0,\infty)$ and $k_0\in[0,\infty)$. Then,  \\
$\displaystyle{\Delta_\ell^m u(k)=\sum\limits_{i=m}^{n-1}\frac{(k-k_0)_\ell^{(i-m)}}{(i-m)!\ell^{(i-m)}}\Delta_\ell^iu(k_0+j)}$
\begin{equation}\label{1.8.7}
+\sum\limits_{r=0}^{\frac{k-k_0-j}{\ell}-n+m}\frac{(k-k_0-j-r\ell-\ell)_\ell^{(n-m-1)}}{(n-m-1)!\ell^{n-m-1}}\Delta_\ell^nu(k_0+j+r\ell),
\end{equation}
where $k_\ell^{(n)}=k(k-\ell)(k-2\ell)\dots(k-(n-1)\ell)$ for all $k\in[k_0,\infty)$, $j=k-k_0-\left[\frac{k-k_0}{\ell}\right]\ell$ and $0\leq m\leq n-1$.
\end{lemma}

\begin{lemma}\label{le1.8.10} \cite{Msm2012}
Let $1\leq m\leq n-1$  and $u(k)$ be defined on $\mathbb{N}_{\ell}(k_0)$. Then,
\begin{enumerate}
	\item $\displaystyle{\liminf_{k\to\infty}\Delta_{\ell}^{m}u(k)>0}$ ~~\text{implies}~~ $\displaystyle{\lim_{k\to\infty}\Delta_{\ell}^{i}u(k)=\infty}$, $0\leq i \leq m-1$.
	\item $\displaystyle{\limsup_{k\to\infty}\Delta_{\ell}^{m}u(k)<0}$ ~~\text{implies}~~ $\displaystyle{\lim_{k\to\infty}\Delta_{\ell}^{i}u(k)=-\infty}$, $0\leq i \leq m-1$.
\end{enumerate}
\end{lemma}
\begin{lemma} \label{th1.8.11}Let $u(k)$ be defined on $\mathbb{N}_{\ell}(k_0)$, and $u(k)>0$ with $\Delta_{\ell}^{n}(k)$ is of constant sign on $\mathbb{N}_{\ell}(k_0)$ and not zero. Then, there $\exists$ an integer $m$, $0\leq m\leq n$ with $n+m$ odd for 
$\Delta_{\ell}^{n}(k)\leq 0$ or $n+m$ even for $\Delta_{\ell}^{n}(k)\geq 0$ and such that $m \geq 1$ implies
\begin{equation}\label{eq2.2}
	\Delta_{\ell}^{i}u(k) > 0 ~~\text{for all large  }~~k\in \mathbb{N}_{\ell}(k_0), 1\leq i \leq m-1
\end{equation}
and $m\leq n-1$ implies
 \begin{equation}\label{eq2.1}
	 (-1)^{m+i}\Delta_{\ell}^{i}u(k) > 0 ~~\text{for  all }~~k\in \mathbb{N}_{\ell}(k_0), m\leq i \leq n-1.
\end{equation}
for all large $n\in\mathbb{N}_{\ell}(k_0)$ and $n\geq N$.
\end{lemma}
\begin{proof}
There are two possible cases to consider.\\
{\bf Case 1.} $\Delta_{\ell}^{m}u(k)\leq 0$ on $\mathbb{N}_{\ell}(k_0)$. First we prove that $\Delta_{\ell}^{n-1}u(k)>0$ on $\mathbb{N}_{\ell}(k_0)$. If not, there $\exists$ some $k_1 \geq k_0$ in $\mathbb{N}_{\ell}(k_0)$ such that $\Delta_{\ell}^{n-1}u(k_1)\leq 0$. As  $\Delta_{\ell}^{n-1}u(k)>0$ is decreasing and not identically constant on $\mathbb{N}_{\ell}(k_0)$, there $\exists$ $k_2\in \mathbb{N}_{\ell}(k_1)$ such that $\Delta_{\ell}^{n-1}u(k)\leq \Delta_{\ell}^{n-1}u(k_2)\leq \Delta_{\ell}^{n-1}u(k_1)\leq 0$ for all $k\in \mathbb{N}_{\ell}(k_2)$, But, from Lemma \ref{le1.8.10} we find $\lim_{k\to\infty}u(k)=-\infty$ which is a contradicts to $u(k)>0$. Thus, $\Delta_{\ell}^{n-1}u(k)>0$ on $\mathbb{N}_{\ell}(k_0)$ and there $\exists$ a smallest integer $m$, $0\leq m\leq n-1$ with $n+m$ odd and 
\begin{equation}\label{eq1.8.11}
	(-1)^{m+i}\Delta_{\ell}^{i}u(k)>0 ~~\text{on}~~\mathbb{N}_{\ell}(k_0), m\leq i\leq n-1.
\end{equation}
Next let $m>1$ and 
\begin{equation}\label{eq1.8.12}
	\Delta_{\ell}^{m-1}u(k)<0 ~~\text{on}~~\mathbb{N}_{\ell}(k_0),
\end{equation}
then once again form Lemma \ref{le1.8.10} it follows that 
\begin{equation}\label{eq1.8.13}
	\Delta_{\ell}^{m-2}u(k)<0 ~~\text{on}~~\mathbb{N}_{\ell}(k_0).
\end{equation}
Inequalities (\ref{eq1.8.11})-(\ref{eq1.8.13}) can be unified to 
\begin{equation}\label{eq1.8.14}
	(-1)^{(m-2)+i}\Delta_{\ell}^{i}u(k)>0 ~~\text{on}~~\mathbb{N}_{\ell}(k_0), m-2\leq i \leq n-1,
\end{equation}
which is contrary to the definition of $m$. So, (\ref{eq1.8.12}) fails and $\Delta_{\ell}^{m-1}u(k)\geq 0$ on $\mathbb{N}_{\ell}(a)$. From (\ref{eq1.8.11}), $\Delta_{\ell}^{m-1}u(k)$ is non-decreasing and hence $\lim_{k\to\infty}\Delta_{\ell}^{m-1}u(k)>0$. If $m>2$, we find from Lemma \ref{le1.8.10} that $\lim_{k\to\infty}\Delta_{\ell}^{i}u(k)=\infty$, $1 \leq i \leq m-2$. Thus, $\Delta_{\ell}^{i}u(k)> 0$ for all large $k\in\mathbb{N}_{\ell}(k_0)$, $1\leq i \leq m-1$.\\
{\bf Case 2.} $\Delta_{\ell}^{n}u(k)\geq 0$ on $\mathbb{N}_{\ell}(k_0)$. Let $k_3\in\mathbb{N}_{\ell}(k_2)$ be such that $\Delta_{\ell}^{n-1}u_{k_3}\geq 0$, then since $\Delta_{\ell}^{n-1}u(k)$ is non-decreasing and not identically constant. There exists some $k_4\in\mathbb{N}_{\ell}(k_3)$ such that $\Delta_{\ell}^{n-1}u(k)> 0$ for all $k\in\mathbb{N}_{\ell}(k_4)$. Thus, $\lim_{k\to\infty}\Delta_{\ell}^{n-1}u(k)> 0$ and from Lemma \ref{le1.8.10} $\lim_{k\to\infty}\Delta_{\ell}^{i}u(k) =\infty$, $1\leq i \leq n-2$ and so $\Delta_{\ell}^{i}u(k)> 0$ for all large $k$ in $\mathbb{N}_{\ell}(k_0)$, $1 \leq i \leq n-1$. This proves the theorem $m=n$. In case $\Delta_{\ell}^{n-1}u(k)< 0$ for all $k\in \mathbb{N}_{\ell}(k_0)$, we find from Lemma \ref{le1.8.10} that $\Delta_{\ell}^{n-2}u(k)> 0$ for all  $k\in\mathbb{N}_{\ell}(k_0)$. The proof will be as in Case 1.
\end{proof}
\begin{lemma}\label{lem2}
Let $u(k)$ be defined on $\mathbb{N}_{\ell}(k_0)$, and $u_k>0$ with $\Delta_{\ell}^{n}u(k) \leq 0$  on $\mathbb{N}_{\ell}(k_0)$ and not zero. Then, for a large integer $k_1$ in $\mathbb{N}_{\ell}(k_0)$  and $k\in \mathbb{N}_{\ell}(k_1)$
\begin{equation}
	u(k)\geq \frac{(k-k_1)^{(n-1)}}{(n-1)!\ell^{n-1}}\Delta_{\ell}^{n-1}u(2^{n-2}k)
\end{equation}
where $k_{\ell}^{(n)}=k(k-\ell)(k-2\ell)\cdots(k-(n-1)\ell)$. Now if $\{u(n)\}$ is increasing, then
\begin{equation}
	u(k)\geq \frac{1}{(n-1)!\ell^{n-1}}\left(\frac{k}{2^{n-2}}\right)^{(n-1)}\Delta_{\ell}^{n-1}u(k) ~~\text{for all}~~ k\geq2^{n-2}k.
\end{equation}
\end{lemma}
\begin{proof}
Lemma \ref{th1.8.11} follows that $(-1)^{n+i}\Delta_{\ell}^{i}u(k)>0$ and $\Delta_{\ell}^{i}u(k)>0$ for large $k$ in $\mathbb{N}_{\ell}(k_0)$, say, for all $k\geq k_1$ in $\mathbb{N}_{\ell}(k_0)$, $1\leq i \leq m-1$. Using these inequalities, we obtain
\begin{flalign*}
	-\Delta_{\ell}^{n-2}u(k)&=-\Delta_{\ell}^{n-2}u(\infty)+\sum_{r=0}^{\infty}\Delta_{\ell}^{n-1}u(k+r\ell)\\
&\geq \sum_{r=0}^{\frac{k}{\ell}}\Delta_{\ell}^{n-1}u(k+r\ell) \geq \frac{1}{\ell}\Delta_{\ell}^{n-1}u(2k)(k)_{\ell}^{(1)}\\
\Delta_{\ell}^{n-3}u(k)&= \Delta_{\ell}^{n-3}u(\infty)-\sum_{r=0}^{\infty}\Delta_{\ell}^{n-2}u(k+r\ell) \\
&\geq \frac{1}{\ell}\sum_{r=0}^{\frac{k}{\ell}}(k+r\ell)_{\ell}^{(1)}\Delta_{\ell}^{n-1}u(2(k+r\ell))
\geq \Delta_{\ell}^{n-1}u(2^2k)\frac{1}{2!\ell^2}(k)_{\ell}^{(2)}\\
&\cdots\quad\cdots\quad\cdots\\
\Delta_{\ell}^mu(k)&\geq \Delta_{\ell}^{n-1}u(2^{n-m-1}k)\frac{1}{(n-m-1)!\ell^{n-m-1}}(k)_{\ell}^{(n-m-1)}.
\end{flalign*}
Next, we get
\begin{flalign*}
\Delta_\ell^{m-1}u(k)&=\Delta_\ell^{m-1}u(k_1+j)+\sum\limits_{r=0}^{\frac{k-k_1-j-\ell}{\ell}}\Delta_\ell^mu(k_1+j+r\ell)\\
	&\geq  \sum\limits_{r=0}^{\frac{k-k_1-j-\ell}{\ell}}\frac{(k_1+j+r\ell)^{(n-m-1)}}{(n-m-1)!\ell^{n-m-1}}\Delta_{\ell}^{n-1}u(2^{n-m-1}(k_1+j+r\ell))\\
	&\geq  \frac{(k-k_1)^{(n-m)}}{(n-m)!\ell^{n-m}}\Delta_{\ell}^{n-1}u(2^{n-m-1}(k)),
\end{flalign*}
which completes the proof.
\end{proof}

\begin{lemma}\label{lem3}
Consider  that 
\begin{equation}
	\sum\limits_{r=0}^{\infty}\frac{1}{a(k+j+r\ell)}=\infty
\end{equation}
and let $\{x(k)\}$ be a positive solution for  $(\ref{2e})$. Then $\exists$ $k_1\geq k_0$ such that 
\begin{equation*}
	z(k)>0, \Delta_{\ell}z(k)>0, \Delta^{m-1}_{\ell} z(k)>0 ~~\text{and}~~\Delta^{m}_{\ell} z(k)\leq 0~~\text{for all}~~ k\geq k_1
\end{equation*}
\end{lemma}
\begin{proof}
Since $\{x(k)\}$ is a positive solution for eq.  (\ref{2e}) there is $k\geq k_0$ such that $x(k)>0$ and $x_{\tau(k)}>0$ for all $k\geq k_1$. Then by the definition of $z(k)$, we have $z(k)>0$ for all $k\geq k_1$. From the equation $(\ref{c02})$, we also have
\begin{equation}\label{eq2.4}
	\Delta_{\ell}\left(a(k)\Delta_{\ell}^{m-1}z(k)\right)=-q(k)f(x(k)-\rho\ell)<0 ~~\text{for all}~~k\geq k_0.
\end{equation}
Therefore $a(k)\Delta_{\ell}^{m-1}z(k)<0$ for all $k\geq k_1$. Since $a(k)>0$, eventually we have either $\Delta_{\ell}^{m-1}z(k)<0$  or $\Delta_{\ell}^{m-1}z(k)>0$. We shall prove that $\Delta_{\ell}^{m-1}z(k)>0$. If not, then $\exists$ $c<0$ such that 
\begin{equation*}
	a(k)\Delta_{\ell}^{m-1}z(k)\leq c <0~~\text{for all}~~ k\geq k_1.
\end{equation*}
then it follows
\begin{equation*}
	\Delta_{\ell}^{m-2}z(k)-\Delta_{\ell}^{m-2}z(k_1)  \leq c \sum\limits_{r=0}^{\frac{k-k_1-j-\ell}{\ell}}\frac{1}{a(k_1+j+r\ell)}.
\end{equation*}
Letting $k\to\infty$ in the last inequality, we see that $\Delta_{\ell}^{m-2}z(k)\to -\infty$. That is $\Delta_{\ell}^{m-2}z(k)<0$ eventually. Now $\Delta_{\ell}^{m-2}z(k)<0$  implies $\Delta_{\ell}^{m-3}z(k)<0$. If we Continue this process we get $z(k)<0$, it is a contrary to assumption. Thus $\Delta_{\ell}^{m-1}z(k)>0$. Moreover $\{a(k)\}$ is positive and increasing and $\Delta_{\ell}\left(a(k)\Delta_{\ell}^{m-1}z(k)\right)<0$ for all $k\geq k_1$, we have $\Delta_{\ell}^{m}z(k)\leq0$ for all $k\geq k_1$.
\end{proof}
\begin{lemma}\label{lem4}\cite{goyri01}
The first order generalized difference inequality
\begin{equation*}
	\Delta_{\ell}y(k)+p(k)y(k-\rho\ell)\leq 0
\end{equation*}
eventually has no positive solution if 
\begin{equation}
	\liminf\limits_{k\to\infty}\sum\limits_{r=0}^{\frac{k-(k-\rho\ell)-j-\ell}{\ell}}p(k-\rho\ell+j+r\ell)>\left(\frac{\rho\ell}{\rho\ell+1}\right)^{\rho\ell+1}
\end{equation}
or 
\begin{equation}
	\limsup\limits_{k\to\infty}\sum\limits_{r=0}^{\frac{k-(k-\rho\ell)-j}{\ell}}p(k-\rho\ell+j+r\ell)>1.
\end{equation}
\end{lemma}

\section{Oscillatory Results}

In this section, we introduce a few sufficient conditions for oscillatory solutions for eq. $(\ref{2e})$.  Throughout this study we consider 
\begin{equation*}
	P(k)=\min\{q(k), q(\tau(k))\}, ~~ Q(k) = LP(k), ~~\text{and}~~ \delta(k) = \sum\limits_{r=0}^{\infty}\frac{1}{a(k+j+r\ell)}. 
\end{equation*}
\begin{thm}\label{thm1}
Let  $\delta(k)=\infty$. 
If
\begin{equation}\label{eq3.1}
	\sum\limits_{r=0}^{\infty}P(k+j+r\ell)=\infty,
\end{equation}
then every solution $\{x(k)\}$ of eq. $(\ref{2e})$ is  oscillatory.
\end{thm}
\begin{proof}
Let $\{x(k)\}$ be a non-oscillatory solution for eq. $(\ref{2e})$. We may assume without loss of generality that $\{x(k)\}$ is a positive solution of equation $(\ref{2e})$. Then $\exists$ a $k_1\geq k_0$ such that $x(k)>0$, $x(\tau(k))$ and $x(k-\rho\ell)>0$ for all $k\geq k_1$. Then from Lemma \ref{lem3}, we have $z(k)>0$, $\Delta_{\ell}z(k)>0$, $\Delta_{\ell}^{m-1}z(k)>0$ and $\Delta_{\ell}^{m}z(k)\leq0$ for all $k\geq k_1$.\\
\indent\hspace{1cm} Now, using $(c_4)$ in equation $(\ref{2e})$, we see that
\begin{equation}
	\Delta_{\ell}\left(a(k)\Delta_{\ell}^{m-1}z(k)\right)=-q(k)f(x(k-\rho\ell))\leq -Lq(k)x({k-\rho\ell})<0~~\forall ~~k\geq k_1.
\end{equation}
Therefore $a(k)\Delta_{\ell}^{m-1}z(k)$ is decreasing. Also from the last inequality, we have
\begin{align}
&\Delta_{\ell}\left(a(k)\Delta_{\ell}^{m-1}z(k)\right)+Lq(k)x(k-\rho\ell)+p\Delta_{\ell}\left(a(\tau(k))\Delta_{\ell}^{m-1}z(\tau(k))\right)\nonumber\\
	&\hspace{3.5cm}+Lq(\tau(k))px(\tau(k-\rho\ell)\leq 0,~~\forall~~ k\geq k_1.	
\end{align}
That is 
\begin{equation}\label{eq3.3}
\Delta_{\ell}\left(a(k)\Delta_{\ell}^{m-1}z(k)\right)	+LP(k)z(k-\rho\ell)+p\Delta_{\ell}\left(a(\tau(k))\Delta_{\ell}^{m-1}z(\tau(k))\right)\leq 0.
\end{equation}
Now by summing up from $k_1$ to $k-\ell$, we obtain
\begin{align}
&a(k)\Delta_{\ell}^{m-1}z(k)-a(k_1)\Delta_{\ell}^{m-1}z(k_1)+L\sum\limits_{r=0}^{\frac{k-k_1-j-\ell}{\ell}} P(k_1+j+r\ell)z(k_1+j+r\ell-\rho\ell)\nonumber \\
	&\hspace{1cm}
	+pa(\tau(k))\Delta_{\ell}^{m-1}z(\tau(k))-pa(\tau(k_1))\Delta_{\ell}^{m-1}z(\tau(k_1))
	\leq 0~~\forall~~ k\geq k_1.	
\end{align}
That is 
\begin{align}\label{eq3.4}
&L\sum\limits_{r=0}^{\frac{k-k_1-j-\ell}{\ell}} P(k_1+j+r\ell)z(k_1+j+r\ell-\rho\ell)\leq
 a(k_1)\Delta_{\ell}^{m-1}z(k_1)-a(k)\Delta_{\ell}^{m-1}z(k) 
	\nonumber\\
	&\hspace{1cm}-pa(\tau(k))\Delta_{\ell}^{m-1}z(\tau(k))+pa(\tau(k_1)\Delta_{\ell}^{m-1}z(\tau(k_1)
	\leq 0~~\text{for all}~~ k\geq k_1.	
\end{align}
Since $\Delta_{\ell}z(k)>0$ and $z(k)>0$ there $\exists$ a constant $c\geq 0$ such that $z(k-\rho\ell)\geq c$ for all $k\geq k_1$ and using the monotonicity of $a(k)\Delta_{\ell}z(k)$ in the last inequity and letting $k\to\infty$, we get
\begin{equation}\label{eq3.5}
	L\sum\limits_{r=0}^{\frac{k-k_1-j-\ell}{\ell}} P(k_1+j+r\ell)z(k_1+j+r\ell-\rho\ell)<\infty,
\end{equation}
which leads to contradiction with $(\ref{eq3.1})$. Thus the proof is complete.
\end{proof}
\begin{exm}
Consider 
\begin{align}\label{eeq1.1}
&\Delta_{\ell}\left(k\Delta_{\ell}^{m-1}(x(k)+(k-1)(x(k+4\ell)))\right)\nonumber\\
&\hspace{2cm}+2^{m-2}\left[4k^2+2\ell(m+\ell) k+(m+1)\ell^2\right]x(k-3\ell)=0,~~k\geq3\ell
\end{align}
where $m\geq2$ is an even integer. Here $p(k)=k-1>0$, $a(k)=k$, $\rho=3$ and $q(k)=2^{m-2}\left[4k^2+2\ell(m+\ell) k+(m+1)\ell^2\right]$. 

Thus all conditions in the Theorem \ref{thm1} are satisfied and is oscillatory. Indeed $\{x(k)\}=\{(-1)^{\left\lceil\frac{k}{\ell}\right\rceil}\}$ is the one oscillatory solution for eq. (\ref{eeq1.1})
\end{exm}
\begin{rem}
In Theorem \ref{thm1} no conditions were imposed on the sequence $\{\tau(k)\}$. That is, $\tau(k)$ could be of delay or advanced type. 
\end{rem}
\begin{thm}\label{thm2}
Consider  $\delta(k)=\infty$ 
and $\tau(k)=k+\tau$. If either 
\begin{equation}\label{eq3.6}
\liminf_{k\to\infty}\sum\limits_{r=0}^{\frac{\rho\ell-j-\ell}{\ell}}\frac{(k-2\rho\ell+j+r\ell)^{m-1}Q((k-\rho\ell)+j+r\ell)}{a(k+j+r\ell-2\rho\ell)}\geq \frac{\beta}{\lambda}\left(\frac{\rho\ell}{1+\rho\ell}\right)^{\rho\ell+1}
\end{equation}
or
\begin{equation}\label{eq3.7}
\limsup_{k\to\infty}\sum\limits_{r=0}^{\frac{\rho\ell-j}{\ell}}\frac{(k+j+r\ell-2\rho\ell)^{m-1}Q(k+j+r\ell-\rho\ell)}{a(k+j+r\ell-2\rho\ell)}\geq \frac{\beta}{\lambda},
\end{equation}
where $\lambda\in(0,1)$ and $\beta=(1+p)(m-1)!\ell^{m-1}$, then the solution $\{x(k)\}$ for eq. (\ref{2e}) is oscillatory.
\end{thm}
\begin{proof}
Now asumme $\{x(k)\}$ is a non-oscillatory solution for equation (\ref{2e}). We can consider without loss of generality that there $\exists$ $k_1\geq k_0$  such that $x(k)>0$, $x(\tau(k))>0$ and $x(k-\rho\ell)>0$ for all $k\geq k_1$. Now proceeding as in the previous theorem, we obtain (\ref{eq3.3}). That is,
\begin{equation}\label{eq3.8}
\Delta_{\ell}\left(a(k)\Delta_{\ell}^{m-1}z(k)\right)	+LP(k)z(k-\rho\ell)+p\Delta_{\ell}\left(a(\tau(k))\Delta_{\ell}^{m-1}z(\tau(k))\right)\leq 0.
\end{equation}
Now, since $\Delta_{\ell}^{m-1}z(k)>0$, $\Delta_{\ell}^{m}z(k)\leq0$, using Lemma \ref{lem2} there $\exists$ $k_2\geq k_1$ such that 
\begin{align}\label{eq3.9}
&\Delta_{\ell}\left(a(k)\Delta_{\ell}^{m-1}z(k)\right)	+Q(k)\frac{1}{(m-1)!\ell^{m-1}}\left(\frac{k-\rho\ell}{2^{m-2}}\right)^{m-1}\Delta_{\ell}^{m-1}z(k-\rho\ell)\nonumber\\
&\hspace{3cm}+p\Delta_{\ell}\left(a(\tau(k))\Delta_{\ell}^{m-1}z(\tau(k))\right)\leq 0. ~~\text{for all}~~ k\geq k_2\geq 2^{m-2}.
\end{align}
Put $u(k)=a(k)\Delta_{\ell}^{m-1}z(k)$. Then $u(k)>0$ and $\Delta_{\ell}u(k)\leq 0$ and the last inequality becomes
\begin{equation}\label{eq3.10}
\Delta_{\ell}\left(u(k)+pu(\tau(k))\right)	+\frac{\lambda Q(k)}{(m-1)!\ell^{m-1}}\left(\frac{(k-\rho\ell)^{m-1}}{a(k-\rho\ell)}\right)u(k-\rho\ell)\leq 0, 
\end{equation}
for all $k\geq k_2$, for every $\lambda$, where $0<\lambda=\left(\dfrac{1}{2^{m-2}}\right)^{m-1}<1$.\\
Now set $w(k)=u(k)+pu(\tau(k))$. Then $w(k)>0$ and since $u(k)$ is decreasing and having $\tau(k)=k+\tau\geq k$, thus we have 
\begin{equation}\label{eq3.11}
w(k)\leq (1+p)u(k).
\end{equation}
Using (\ref{eq3.11}) in (\ref{eq3.10}), we notice that $w(k)$ is a positive solution of 
\begin{equation}\label{eq3.12}
\Delta_{\ell}w(k)	+\frac{\lambda Q(k)}{(m-1)!\ell^{m-1}}\left(\frac{(k-\rho\ell)^{m-1}}{(1+p)a(n-\rho\ell)}\right)w(k-\rho\ell)\leq 0, ~~\text{for all}~~ k\geq k_2.
\end{equation}
Now there are two possibilities either (\ref{eq3.6}) or (\ref{eq3.7}) holds.\\

\textbf{Case(i).} If (\ref{eq3.6}) holds, then by using the Lemma \ref{lem4} we obtain the inequality (\ref{eq3.12}) which has no positive solution, and that is again a contradictory.\\

\textbf{Case(ii).} If the condition (\ref{eq3.7}) holds, by Lemma \ref{lem4} we confirm that the inequality (\ref{eq3.12}) has no positive solution, which inturn is also a contradiction.\\ 
Thus the proof is now completed.
\end{proof}
\begin{exm}
Consider the equation
\begin{align}\label{eeq1.2}
\Delta_{\ell}\left(k\Delta_{\ell}^{m-1}(x(k)+(x(k+2\ell)))\right)+2^{m}(2k+\ell)x(k-3\ell)=0,~~k\geq3\ell
\end{align}
where $m\geq4$ is an even integer and here $p(k)=1>0$, $a(k)=k$, $q(k)=2^{m}(2k+\ell)$, $\tau(k)=k+2\ell$ and $\rho=3$.\\

It is not so difficult to see that all conditions in Theorem \ref{thm2} are satisfied and thus every solution of equation (\ref{eeq1.2}) is oscillatory. Indeed $\{x(k)\}=\{(-1)^{\left\lceil\frac{k}{\ell}\right\rceil}\}$ is a one of such oscillatory solution of equation (\ref{eeq1.2})
\end{exm}

\begin{thm}\label{thm3}
Assume that $\delta(k)=\infty$
%$\displaystyle{\sum\limits_{r=0}^{\infty}\frac{1}{a(k+j+r\ell)}=\infty}$ 
and $k-\rho\ell\leq \tau(k)\leq k$. If either 
\begin{equation}\label{eq3.13}
\liminf_{k\to\infty}\sum\limits_{r=0}^{\frac{\rho\ell-j-\ell}{\ell}}\frac{(k+j+r\ell-2\rho\ell)^{m-1}Q(k+j+r\ell-\rho\ell)}{a(k+j+r\ell-2\rho\ell)}< \beta\left(\frac{\rho\ell}{1+\rho\ell}\right)^{\rho\ell+1}
\end{equation}
or
\begin{equation}\label{eq3.14}
\limsup_{k\to\infty}\sum\limits_{r=0}^{\frac{\rho\ell-j}{\ell}}\frac{(k+j+r\ell-2\rho\ell)^{m-1}Q(k+j+r\ell-\rho\ell)}{a(k+j+r\ell-2\rho\ell)}<\beta ,
\end{equation}
where $\beta=(1+p)(m-1)!\ell^{m-1}$, then every solution $\{x(k)\}$ of equation (\ref{2e}) is oscillatory.
\end{thm}
\begin{proof}
Similiar to the previous proof, we consider $\{x(k)\}$ is a non-oscillatory solution of equation (\ref{2e}). Then assume $\{x(k)\}$ is a positive solution of equation (\ref{2e}). It follows that there is an integer $k_1\geq k_0$  such that $x(k)>0$, $x(\tau(k))>0$ and $x(k-\rho\ell)>0$ for all $k\geq k_1$. Now proceeding as in the previous theorem, we obtain 
\begin{equation}\label{eq3.15}
\Delta_{\ell}\left(u(k)+pu(\tau(k))\right)	+\frac{\lambda Q(k)}{(m-1)!\ell^{m-1}}\left(\frac{(k-\rho\ell)^{m-1}}{a(k-\rho\ell)}\right)u(k-\rho\ell)\leq 0,~~\forall~~ k\geq k_1. 
\end{equation}
Put $w(k)=u(k)+pu(\tau(k))$. Then $w(k)>0$. Since $u(k)$ is decreasing, we have 
\begin{equation}\label{eq3.16}
w(k)=u(k)+pu(\tau(k))\leq (1+p)u(k)~~\text{for}~~\tau(k)\leq k.
\end{equation}
 using (\ref{eq3.16}) in (\ref{eq3.15}), we get
\begin{equation}\label{eq3.17}
\Delta_{\ell}w(k)	+\frac{\lambda Q(k)}{(m-1)!\ell^{m-1}}\left(\frac{(k-\rho\ell)^{m-1}}{(1+p)a(n-\rho\ell)}\right)w(\tau^{-1}(k-\rho\ell))\leq 0, ~~\forall~~ k\geq k_1.
\end{equation}
Thus $\{w(k)\}$ is a positive solution and satisfies the inequality (\ref{eq3.17}). Similarly, we have two cases as follows:

\textbf{Case(i).} If (\ref{eq3.13}) holds, then by using Lemma \ref{lem4} we obtain the inequality (\ref{eq3.17}) which  has no positive solution, it is also a contradiction.\\

\textbf{Case(ii).} If the condition (\ref{eq3.14}) holds, thus Lemma \ref{lem4} confirms that the inequality (\ref{eq3.17}) has no positive solution, is again a contradictory. \\
These complete the proof.
\end{proof}
\begin{exm}
Consider the equation
\begin{align}\label{eeq1.3}
&\Delta_{\ell}\left(k\Delta_{\ell}^{m-1}(x(k)+(k+1)(x(k-\ell)))\right)\nonumber\\
&\hspace{2cm}+2^{m-2}\left[4k^2+(2\ell +2m\ell) k+(m+1)\ell^2\right]x(k-3\ell)=0,~~k\geq3\ell
\end{align}
where $m>3$ is an odd integer. Here $p(k)=k+1>0$, $a(k)=k$, $\rho=3$ and $q(k)=2^{m-2}\left[4k^2+(2\ell +2m\ell) k+(m+1)\ell^2\right]$, $\tau(k)=k-\ell$. \\

The conditions in Theorem \ref{thm3} are satisfied and thus every solution of equation (\ref{eeq1.3}) is oscillatory. Indeed the solution $\{x(k)\}=\{(-1)^{\left\lceil\frac{k}{\ell}\right\rceil}\}$ is oscillatory for (\ref{eeq1.3}).
\end{exm}
\begin{thm}\label{thm4}
Assume that $\delta(k)<\infty$
%$\displaystyle{\sum\limits_{r=0}^{\infty}\frac{1}{a(k+j+r\ell)}=\infty}$
 and $k-\rho\ell\leq \tau(k)\leq k$. If either (\ref{eq3.13}) or when $\tau^{-1}(k-\rho\ell)$ is non-decreasing with (\ref{eq3.14}) holds and for sufficiently large $k_1\geq k_0$
\begin{align}\label{eq3.18}
&\limsup_{k\to\infty}\sum_{r=0}^{\frac{k-k_0-j-\ell}{\ell}}\left[\frac{\lambda\delta(k_0+j+r\ell)Q(k_0+j+r\ell)(k_0+j+r\ell-\rho\ell)^{m-2}}{(m-2)!}\right.\nonumber\\
&\hspace{3cm}\left.-\frac{ (1+p)}{4a(k_0+j+r\ell+\ell)\delta(k_0+j+r\ell+\ell)}\right]=\infty,
\end{align}
then every solution $\{x(k)\}$ of equation (\ref{2e}) is oscillatory.
\end{thm}
\begin{proof}
Let $\{x(k)\}$ be a non-oscillatory and be a positive solution for equation (\ref{2e}). Then there $\exists$ an integer $k_1\geq k_0$  such that $x(k)>0$, $x(\tau(k))>0$ and $x(k-\rho\ell)>0$ for all $k\geq k_1$.  From equation (\ref{2e}) we see that $\Delta_{\ell}\left(a(k)\Delta_{\ell}^{m-1}z(k)\right)\leq 0$ for all $k\geq k_1$. Since $\{a(k)\}$ is positive, $\Delta_{\ell}^{m-1}z(k)$ is of one sign for all $k\geq k_1$.\\

\textbf{Case(i):} Suppose $\Delta_{\ell}^{m-1}z(k)>0$ eventually, the proof is similar to the Case (i) in Theorem \ref{thm3} and hence we omit the details.\\
\textbf{Case(ii):} Suppose $\Delta_{\ell}^{m-1}z(k)<0$ eventually, then by Lemma \ref{th1.8.11}, we have $\Delta_{\ell}^{m-2}z(k)>0$ and $\Delta_{\ell}z(k)>0$. Now define $w(k)$ by
\begin{equation}\label{eq3.19}
w(k)=\frac{a(k)\Delta_{\ell}^{m-1}z(k)}{\Delta_{\ell}^{m-2}z(k)}~~\text{for all}~~k\geq k_2\geq k_1.
\end{equation}
Then $w(k)<0$ and 
\begin{equation*}
\Delta_{\ell}w(k)=\frac{\Delta_{\ell}\left(a(k)\Delta_{\ell}^{m-1}z(k)\right)}{\Delta_{\ell}^{m-2}z(k)}-\frac{a(k+\ell)\Delta_{\ell}^{m-1}z(k+\ell)}{\Delta_{\ell}^{m-2}z(k+\ell)\Delta_{\ell}^{m-2}z(k)}\Delta_{\ell}^{m-1}z(k)~~\text{for all}~~k\geq k_2.
\end{equation*}
Since $a(k)\Delta_{\ell}^{m-1}z(k)$ is decreasing and $\Delta_{\ell}^{m-2}z(k)$ is increasing, we have
\begin{equation}\label{eq3.20}
\Delta_{\ell}w(k)\leq\frac{\Delta_{\ell}\left(a(k)\Delta_{\ell}^{m-1}z(k)\right)}{\Delta_{\ell}^{m-2}z(k)}-\frac{w^2(k+\ell)}{a(k+\ell)}.
\end{equation}
Using the decreasing nature of $a(k)\Delta_{\ell}^{m-1}z(k)$ we have
\begin{equation*}
a(l)\Delta_{\ell}^{m-1}z(l)\leq a(k)\Delta_{\ell}^{m-1}z(k) ~~\text{for all}~~ l\geq k\geq k_2.
\end{equation*}
Divide the above inequality by $a(l)$ and then sum it from $k$ to $l-\ell$, we obtain
\begin{equation*}
\Delta_{\ell}^{m-2}z(l)-\Delta_{\ell}^{m-2}z(k)\leq a(k)\Delta_{\ell}^{m-1}z(k)\sum\limits_{r=0}^{\frac{l-k-j-\ell}{\ell}}\frac{1}{a(k+j+r\ell)}.
\end{equation*}
Letting $l\to\infty$, we obtain
\begin{align}\label{eq3.21}
0&\leq \Delta_{\ell}^{m-2}z(k)+a(k)\Delta_{\ell}^{m-1}z(k)\delta(k)\nonumber\\
\text{or}\hspace{1cm} -1&\leq \frac{a(k)\Delta_{\ell}^{m-1}z(k)\delta(k)}{\Delta_{\ell}^{m-2}z(k)}=w(k)\delta{k}\leq 0 ~~\text{for all}~~ k\geq k_2.
\end{align}
Define  $v(k)$ by
\begin{equation}\label{eq3.22}
v(k)=\frac{a(\tau(k))\Delta_{\ell}^{m-1}z(\tau(k))\delta(k)}{\Delta_{\ell}^{m-2}z(k)}~~ \text{for all}~~ k\geq k_2.
\end{equation}
We have $v(k)\leq 0$ and
\begin{equation}\label{eq3.23}
-1\leq v(k)\delta(k)\leq 0~~ \text{for all}~~ k\geq k_2.
\end{equation}
From (\ref{eq3.22}), we get
\begin{align}\label{eq3.24}
\Delta_{\ell}v(k)&=\frac{\Delta_{\ell}\left(a(\tau(k))\Delta_{\ell}^{m-1}z(\tau(k))\right)}{\Delta_{\ell}^{m-2}z(k)}-\frac{a(\tau(k+\ell))\Delta_{\ell}^{m-1}z(\tau(k+\ell))}{\Delta_{\ell}^{m-2}z(k+\ell)\Delta_{\ell}^{m-2}z(k)}\Delta_{\ell}^{m-1}z(k)\nonumber\\
&\leq \frac{\Delta_{\ell}\left(a(\tau(k))\Delta_{\ell}^{m-1}z(\tau(k))\right)}{\Delta_{\ell}^{m-2}z(k)}-\frac{v^2(k+\ell)}{a(\tau(k+\ell))}.
\end{align}
Combining (\ref{eq3.20}) and (\ref{eq3.24}), we obtain
\begin{align*}
\Delta_{\ell}w(k)+p\Delta_{\ell}v(k)&\leq\frac{\Delta_{\ell}\left(a(k)\Delta_{\ell}^{m-1}z(k)\right)}{\Delta_{\ell}^{m-2}z(k)}-\frac{w^2(k+\ell)}{a(k+\ell)}\\
&\hspace{1cm}+p\frac{\Delta_{\ell}\left(a(\tau(k))\Delta_{\ell}^{m-1}z(\tau(k))\right)}{\Delta_{\ell}^{m-2}z(k)}-p\frac{v^2(k+\ell)}{a(\tau(k+\ell))}.
\end{align*}
Using  (\ref{eq3.3}) in the last inequality, we have
\begin{equation}\label{eq3.25}
\Delta_{\ell}w(k)+p\Delta_{\ell}v(k)\leq\frac{-LP(k)z(k-\rho\ell)}{\Delta_{\ell}^{m-2}z(k)}-\frac{w^2(k+\ell)}{a(k+\ell)}-p\frac{v^2(k+\ell)}{a(\tau(k+\ell))}.
\end{equation}
Now from Lemma \ref{lem2} we obtain
\begin{equation}\label{eq3.26}
z(k-\rho\ell)\geq \frac{\lambda}{(m-2)!\ell^{m-2}}(k-\rho\ell)^{(m-2)}\Delta_{\ell}^{m-2}z(k-\rho\ell).
\end{equation}
Since $\Delta_{\ell}^{m-1}z(k)<0$ and $k-\rho\ell\leq k$, we have
\begin{equation}\label{eq3.27}
\Delta_{\ell}^{m-2}z(k)<\Delta_{\ell}^{m-2}z(k-\rho\ell) .
\end{equation}
Combining the inequalities (\ref{eq3.25}), (\ref{eq3.26}) and (\ref{eq3.27}), we have
\begin{equation}\label{eq3.28}
\Delta_{\ell}w(k)+p\Delta_{\ell}v(k)\leq\frac{-\lambda Q(k)(k-\rho\ell)^{m-2}}{(m-2)!\ell^{m-2}}-\frac{w^2(k+\ell)}{a(k+\ell)}-p\frac{v^2(k+\ell)}{a(k+\ell)}.
\end{equation}
Multiplying (\ref{eq3.28}) by $\delta(k)$ and summation is taken on the resulting inequality from $k_2$ to $k-\ell$, we obtain
\begin{flalign}\label{eq3.29}
&\delta(k)w(k)-\delta(k_2)w(k_2) + \sum_{r=0}^{\frac{k-k_2-j-\ell}{\ell}}\frac{w(k_2+j+r\ell+\ell)}{a(k_2+j+r\ell)}+p\delta(k)v(k)-p\delta(k_2)v(k_2)\nonumber\\
&+p\sum_{r=0}^{\frac{k-k_2-j-\ell}{\ell}}\frac{v(k_2+j+r\ell+\ell)}{a(k_2+j+r\ell)} +\sum_{r=0}^{\frac{k-k_2-j-\ell}{\ell}}\frac{w^2(k_2+j+r\ell+\ell)}{a(k_2+j+r\ell+\ell)}\delta(k_2+j+r\ell)\nonumber\\
&+\frac{\lambda}{(m-2)!\ell^{m-2}} \sum_{r=0}^{\frac{k-k_2-j-\ell}{\ell}}Q(k_2+j+r\ell)(k_2+j+r\ell-\rho\ell)^{m-2}\delta(k_2+j+r\ell) \nonumber\\
&+p\sum_{r=0}^{\frac{k-k_2-j-\ell}{\ell}}\frac{v^2(k_2+j+r\ell+\ell)}{a(k_2+j+r\ell+\ell)}\delta(k_2+j+r\ell)\leq 0.
\end{flalign}
By nature $\{a(k)\}$ increasing, and $\{\delta(k)\}$ decreasing thus the completion of square yields
\begin{flalign*} %\label{eq3.29}
&\delta(k)w(k)-\delta(k_2)w(k_2) + p\delta(k)v(k)-p\delta(k_2)v(k_2)\nonumber\\
&\hspace{1cm}+\frac{\lambda}{(m-2)!\ell^{m-2}} \sum_{r=0}^{\frac{k-k_2-j-\ell}{\ell}}Q(k_2+j+r\ell)(k_2+j+r\ell-\rho\ell)^{m-2}\delta(k_2+j+r\ell)\nonumber\\
&\hspace{1cm}-\frac{1}{4}\sum_{r=0}^{\frac{k-k_2-j-\ell}{\ell}}\frac{1}{a(k_2+j+r\ell+\ell)\delta(k_2+j+r\ell+\ell)}\nonumber\\
&\hspace{1cm}-\frac{p}{4}\sum_{r=0}^{\frac{k-k_2-j-\ell}{\ell}}\frac{1}{a(k_2+j+r\ell+\ell)\delta(k_2+j+r\ell+\ell)}  \leq 0.\nonumber\\
&\text{or} \\
&\delta(k)w(k) + p\delta(k)v(k)\nonumber\\
&\hspace{0.5cm}+\sum_{r=0}^{\frac{k-k_2-j-\ell}{\ell}}\left[\frac{\lambda Q(k_2+j+r\ell)(k_2+j+r\ell-\rho\ell)^{m-2}\delta(k_2+j+r\ell)}{(m-2)!\ell^{m-2}}\right.\nonumber\\
&\hspace{2cm}\left. -\frac{1+p}{4a(k_2+j+r\ell+\ell)\delta(k_2+j+r\ell+\ell)}\right] \leq  \delta(k_2)w(k_2)+p\delta(k_2)v(k_2).&{}
\end{flalign*}
When we take limit supremum as $k\to\infty$ in the last inequality, we deduce a contradiction to (\ref{eq3.18}). This completes the proof.
\end{proof}
\begin{exm}
Consider the equation
\begin{align}\label{eeq1.4}
&\Delta_{\ell}\left(k(k+\ell)\Delta_{\ell}^{m-1}(x(k)+4(x(k-2\ell)))\right)\nonumber\\
&\hspace{3cm}+5(2^{m})(k+\ell)^2 x(k-6\ell)=0,~~k\geq6\ell
\end{align}
where $m\geq5$ is an odd integer. Here $p(k)=4>0$, $a(k)=k(k+\ell)$, $\rho=6$ and $q(k)=5(2^{m})(k+\ell)^2$, $\tau(k)=k-2\ell$. Then all conditions in Theorem \ref{thm4} are satisfied and every solution oscillatory, indeed $\{x(k)\}=\{(-1)^{\left\lceil\frac{k}{\ell}\right\rceil}\}$ is oscillatory.
\end{exm}

\begin{thm}\label{thm5}
Let $\delta(k)<\infty$
and $ \tau(k)\geq k$. If either (\ref{eq3.6}) holds or $\tau^{-1}(k-\rho\ell)$ is non-decreasing with (\ref{eq3.7}) holds and for sufficiently large $k_1\geq k_0$
\begin{align}\label{eq3.30}
&\limsup_{k\to\infty}\sum_{r=0}^{\frac{k-k_0-j-\ell}{\ell}}\left[\frac{\lambda\delta(k_0+j+r\ell)Q(k_0+j+r\ell)(k_0+j+r\ell-\rho\ell)^{m-2}\delta(\tau(k_0+j+r\ell))}{(m-2)!\ell^{m-2}}\right.\nonumber\\
&\hspace{4cm}\left.-\frac{ (1+p)}{4a(k_0+j+r\ell+1)\delta(\tau(k_0+j+r\ell+1))}\right]=\infty,
\end{align}
where $0<\lambda<1$ is a constant, and every solution $\{x(k)\}$ of equation (\ref{2e}) is oscillatory.
\end{thm}
\begin{proof}
Consider $\{x(k)\}$ be a non-oscillatory solution of equation (\ref{2e}). We shall prove the case when $\{x(k)\}$ is positive as for the case when $\{x(k)\}$ negative is similar. Since $\{x(k)\}$ is positive there $\exists$  an integer $k_1\geq k_0$  such that $x(k)>0$, $x(\tau(k))>0$ and $x(k-\rho\ell)>0$ for all $k\geq k_1$.  From equation (\ref{2e}), we see that $\{a(k)\Delta_{\ell}^{m-1}z(k)\}$ is decreasing for all $k\geq k_1$. Then there are two cases for $\Delta_{\ell}^{m-1}z(k)$, namely, either $\Delta_{\ell}^{m-1}z(k)>0$ eventually or $\Delta_{\ell}^{m-1}z(k)<0$ eventually.\\
\textbf{Case(i).} Suppose $\Delta_{\ell}^{m-1}z(k)>0$ for all $k\geq k_1$, the proof is similar to the Case (i) in Theorem \ref{thm2} and thus the details are omitted.\\
\textbf{Case(ii).} Suppose $\Delta_{\ell}^{m-1}z(k)<0$ for all $k\geq k_1$, Then Lemma \ref{th1.8.11}, we have $\Delta_{\ell}^{m-2}z(k)>0$ and $\Delta_{\ell}z(k)>0$. Now define $\gamma(k)$ by
\begin{equation}\label{eq3.31}
\gamma(k)=\frac{a(\tau(k))\Delta_{\ell}^{m-1}z(\tau(k))}{\Delta_{\ell}^{m-2}z(k)}~~\text{for all}~~k\geq k_2\geq k_1.
\end{equation}
Then $\gamma(k)<0$ for all $k\geq k_2$. Since $a(k)\Delta_{\ell}^{m-1}z(k)$ is decreasing we have 
\begin{equation*}%\label{eq3.31}
a(\tau(l))\Delta_{\ell}^{m-1}z(\tau(l))\leq a(\tau(k))\Delta_{\ell}^{m-1}z(\tau(k))~~\text{for all}~~ l\geq k\geq k_2.
\end{equation*}
Divide the last inequality by $a(\tau(l))$ and sum it from $k$ to $l-\ell$, we obtain
\begin{equation}\label{eq3.32}
\Delta_{\ell}^{m-2}z(\tau(l))-\Delta_{\ell}^{m-2}z(k)\leq a(\tau(k))\Delta_{\ell}^{m-1}z(\tau(k))\sum_{r=0}^{\frac{l-k-j-\ell}{\ell}}\frac{1}{a(\tau(k+j+r\ell))}.
\end{equation}
Letting $l\to\infty$, we obtain

\begin{equation}\label{eq3.33}
0\leq \Delta_{\ell}^{m-2}z(\tau(k))\leq a(\tau(k))\Delta_{\ell}^{m-2}z(\tau(k))\delta(\tau(k)).
\end{equation}
Since $\Delta_{\ell}^{m-1}z(k)<0$, $\Delta_{\ell}^{m-1}z(k)$ is decreasing and therefore for $\tau(k)\geq k$,  we have
\begin{equation}\label{eq3.34}
\Delta_{\ell}^{m-2}z(\tau(k))\leq \Delta_{\ell}^{m-2}z(k) .
\end{equation}
Combining the inequalities (\ref{eq3.33}), (\ref{eq3.34}) and (\ref{eq3.27}), we have
\begin{equation}\label{eq3.35}
-1\leq v(k)\delta(\tau(k))\leq 0 ~~\text{for all}~~k\geq k_2.
\end{equation}
Similarly defining $w(k)$ by
\begin{equation*}
w(k)=\frac{a(k)\Delta_{\ell}^{m-1}z(k)}{\Delta_{\ell}^{m-2}z(k)}~~\text{for all}~~k\geq k_2.
\end{equation*}
we get
\begin{equation}\label{eq3.36}
-1\leq w(k)\delta(\tau(k))\leq 0 ~~\text{for all}~~ k\geq k_2.
\end{equation}
Based on the proof of Theorem \ref{thm4} we obtain (\ref{eq3.28}). Multiply (\ref{eq3.28}) by $\delta(\tau(k))$ and then sum it form $k_2$ to $k-\ell$, we get
\begin{align}\label{eq3.37}
&\delta(\tau(k))w(k)-\delta(\tau(k_2))w(k_2) + \sum_{r=0}^{\frac{k-k_2-j-\ell}{\ell}}\frac{w(k_2+j+r\ell+\ell)}{a(k_2+j+r\ell)}+p\delta(\tau(k))v(k)\nonumber\\
&-p\delta(\tau(k_2))v(k_2)+p\sum_{r=0}^{\frac{k-k_2-j-\ell}{\ell}}\frac{v(k_2+j+r\ell+\ell)}{a(\tau(k_2+j+r\ell))} +\sum_{r=0}^{\frac{k-k_2-j-\ell}{\ell}}\frac{w^2(k_2+j+r\ell+\ell)}{a(\tau(k_2+j+r\ell))}\nonumber\\
&+\frac{\lambda}{(m-2)!\ell^{m-2}} \sum_{r=0}^{\frac{k-k_2-j-\ell}{\ell}}Q(k_2+j+r\ell)(k_2+j+r\ell-\rho\ell)^{m-2}\delta(\tau(k_2+j+r\ell))\nonumber\\
&\hspace{2cm}+p\sum_{r=0}^{\frac{k-k_2-j-\ell}{\ell}}\frac{v^2(k_2+j+r\ell+\ell)}{a(\tau(k_2+j+r\ell))} \leq 0.
\end{align}
Since $\{a(k)\}$ increasing, and $\{\delta(k)\}$ decreasing  and by using the completion of square, we arrive at
\begin{align*} %\label{eq3.29}
&\delta(k)w(k)-\delta(k_2)w(k_2) + p\delta(k)v(k)-p\delta(k_2)v(k_2)\nonumber\\
&\hspace{1cm}+\frac{\lambda}{(m-2)!\ell^{m-2}} \sum_{r=0}^{\frac{k-k_2-j-\ell}{\ell}}Q(k_2+j+r\ell)(k_2+j+r\ell-\rho\ell)^{m-2}\delta(\tau(k_2+j+r\ell))\nonumber\\
&\hspace{1cm}-\frac{1}{4}\sum_{r=0}^{\frac{k-k_2-j-\ell}{\ell}}\frac{1}{a(k_2+j+r\ell+\ell)\delta(\tau(k_2+j+r\ell+\ell))}\nonumber\\
&\hspace{1cm}-\frac{p}{4}\sum_{r=0}^{\frac{k-k_2j-\ell}{\ell}}\frac{1}{a(k_2+j+r\ell+\ell)\delta(\tau(k_2+j+r\ell+\ell))}  \leq 0.\nonumber\\
&\text{or} \\
&\delta(\tau(k))w(k) + p\delta(\tau(k))v(k)\nonumber\\
&+\sum_{r=0}^{\frac{k-k_2-j-\ell}{\ell}}\left[\frac{\lambda Q(k_2+j+r\ell)(k_2+j+r\ell-\rho\ell)^{m-2}\delta(k_2+j+r\ell)}{(m-2)!\ell^{m-2}}\right.\nonumber\\
&\hspace{1cm}\left. -\frac{1+p}{4a(\tau(k_2+j+r\ell+\ell))\delta(\tau(k_2+j+r\ell+\ell))}\right] \leq  \delta(\tau(k_2))w(k_2)+p\delta(\tau(k_2))v(k_2).
\end{align*}
By taking limit as $k\to\infty$ in the last inequality and arrive at a result which is contrary to (\ref{eq3.18}). That completes the proof.
\end{proof}
\begin{exm}
Consider the equation
\begin{align}\label{eeq1.5}
&\Delta_{\ell}\left(k(k-\ell)\Delta_{\ell}^{m-1}(x(k)+2(x(k+2\ell)))\right)%\nonumber\\&\hspace{3cm}
+3(2^m)k^2 x(k-\ell)=0,~~k\geq 2\ell
\end{align}
where $m\geq4$ is an even integer. Here $p(k)=2>0$, $a(k)=k(k-\ell)$, $q(k)=3(2^m)k^2$, $\tau(k)=k-2\ell$ and $\rho=2$. Since all conditions in Theorem \ref{thm5} are satisfied, every solution of equation (\ref{eeq1.5}) is oscillatory. Indeed $\{x(k)\}=\{(-1)^{\left\lceil\frac{k}{\ell}\right\rceil}\}$ is the such oscillatory solution.
\end{exm}

\section{Conclusion} 

The difference equations are interesting part of mathematics however dealing with $n^{th}$ order difference equations is a tough task. In particular, discussing the oscillation of solutions of ordinary difference equations is also not easy to handle. Further, the delay and neutral difference equations it is still more complicated. Not many results involving the oscillation of solutions of higher order difference equations for both odd and even order difference equations involving $\Delta$ are available. Results for higher order difference equations for the generalized difference operator $\Delta_{\ell}$ are very rare and researchers studying the oscillatory and asymptotic behavior of solutions for higher order neutral difference equations involving $\Delta$ difference operator. These equations naturally arise in the applications such as in population dynamics. In the present work, we generalize the early results by using the generalized difference operator $\Delta_{\ell}$. Many authors studied in order to find sufficient conditions that will ensure all solutions oscillatory. In this work we have identified sufficient conditions for the solutions to be oscilatory.

%********************************************

\end{document}